\documentclass[12pt]{amsart}

\usepackage{amsmath,amssymb,amsthm,amsfonts,amscd,mathrsfs,mathtools, url}

\usepackage{stmaryrd}

\usepackage{tikz-cd}

\usepackage{hyperref}

\DeclareMathAlphabet{\mathpzc}{OT1}{pzc}{m}{it}

\title{faithful tropicalization of hypertoric varieties} 
\author{max b. kutler}
\newcommand{\myemail}{kutler@uoregon.edu}
\date{}

\newcommand{\CC}{\mathbb{C}}
\newcommand{\RR}{\mathbb{R}}

\newcommand{\ZZ}{\mathbb{Z}}

\renewcommand{\AA}{\mathbb{A}}
\newcommand{\PP}{\mathbb{P}}
\newcommand{\TT}{\mathbb{T}}
\newcommand{\GG}{\mathbb{G}}

\newcommand{\fM}{\mathfrak{M}}
\newcommand{\fB}{\mathfrak{B}}
\newcommand{\cB}{\mathcal{B}}
\newcommand{\cF}{\mathcal{F}}

\newcommand{\cR}{\mathcal{R}}
\newcommand{\cT}{\mathcal{T}}

\newcommand{\cA}{\mathcal{A}}

\newcommand{\id}{\operatorname{id}}

\newcommand{\sgn}{\operatorname{sgn}}

\newcommand{\Hom}{\operatorname{Hom}}

\newcommand{\ev}{\operatorname{ev}}

\newcommand{\codim}{\operatorname{codim}}
\newcommand{\Spec}{\operatorname{Spec}}

\newcommand{\trop}{\operatorname{trop}}
\newcommand{\Trop}{\operatorname{Trop}}
\newcommand{\an}{\operatorname{an}}

\newcommand{\relint}{\operatorname{relint}}

\newcommand{\Gr}{\operatorname{Gr}}
\newcommand{\rk}{\operatorname{rk}}
\newcommand{\trunc}{\operatorname{trunc}}

\numberwithin{equation}{section}

\newtheorem{thm}{Theorem}[section]
\newtheorem*{introhtvcones}{Theorem \ref{thm:htvcones}}
\newtheorem*{introfaithful}{Theorem \ref{thm:faithful}}
\newtheorem{lem}[thm]{Lemma}
\newtheorem{prop}[thm]{Proposition}
\newtheorem{cor}[thm]{Corollary}
\theoremstyle{definition}

\newtheorem{rem}[thm]{Remark}
\newtheorem{ex}[thm]{Example}

\begin{document}

\begin{abstract}
The hypertoric variety $\fM_{\cA}$ defined by an affine arrangement $\cA$ admits a natural tropicalization, induced by its embedding in a Lawrence toric variety.  We explicitly describe the polyhedral structure of this tropicalization. Using a recent result of Gubler, Rabinoff, and Werner, we prove that there is a continuous section of the tropicalization map.
\end{abstract}

\maketitle

\section{Introduction}

In this paper, we study the tropicalization of the hypertoric variety $\fM_{\cA}$ defined by an arrangement $\cA$ of affine hyperplanes. Hypertoric varieties were first studied by Bielawski and Dancer \cite{BD00}. They are ``hyperk\"{a}hler analogues'' of toric varieties, and examples of conical symplectic resolutions. The relationship between the variety $\fM_{\cA}$ and the arrangement $\cA$ is analogous to that between a semiprojective toric variety and its polyhedron. See, e.g., \cite{P06} for an overview of this relationship. The hypertoric variety $\fM_{\cA}$ is not, in general, a toric variety. However, it is naturally defined as a closed subvariety of a toric variety, the Lawrence toric variety $\fB_{\cA}$. The Lawrence embedding allows us to define a tropicalization of $\fM_{\cA}$.

Given a closed embedding of a variety $X$ in a toric variety, there is a corresponding tropicalization $\Trop(X)$, which is the continuous image of the Berkovich space $X^{\an}$ under the tropicalization map. The tropicalization may be endowed with the structure of a finite polyhedral complex. A single variety $X$ may yield many distinct tropicalizations, each given by a different choice of embedding into a toric variety. When we speak of \emph{the} tropicalization of $X$, it is always with respect to a chosen embedding.

By a result of Foster, Gross, and Payne \cite{P09, FGP14}, if $X$ has at least one embedding into a toric variety, then the inverse system of all such embeddings induces an inverse system of tropicalizations, and the limit of this system in the category of topological spaces is $X^{\an}$. This raises the question of how well a particular tropicalization approximates the geometry of the analytic space. To this end, a tropicalization is said to be \textbf{faithful} if it admits a continuous section to the tropicalization map $X^{\an} \to \Trop(X)$, and thus $\Trop(X)$ is homeomorphic to a closed subset of $X^{\an}$.

If $X$ is embedded in a torus, as opposed to a more general toric variety, then $\Trop(X)$ is the support of a polyhedral complex, which is balanced when the polyhedra are weighted by tropical multiplicity.  Gubler, Rabinoff, and Werner have proved that the tropicalization is faithful if all tropical multiplicities are equal to one \cite{GRW14}. This was first proved in the case where $X$ is a curve by Baker, Payne, and Rabinoff \cite{BPR11}, who further showed that the tropicalization map is in fact an isometry on finite subgraphs of the hyperbolic interior of $X^{\an}$.

In the more general situation where $X$ is embedded in a toric variety, tropical multiplicity one is no longer sufficient to imply faithfulness: while $\Trop(X)$ may be computed orbit-by-orbit, the continuous sections defined on each stratum may not glue to a continuous section on all of $\Trop(X)$ \cite[Example 8.11]{GRW15}. However, Gubler, Rabinoff, and Werner \cite[Theorem 8.14]{GRW15} have recently proved that if $X$ is embedded in a toric variety with dense torus $T$, then the resulting tropicalization is faithful if the following conditions are satisfied:
\begin{itemize}
\item $X \cap T$ is dense in $X$;
\item the intersection of $X$ with each torus orbit is either empty or equidimensional;
\item $\Trop(X)$ has multiplicity one everywhere;
\item There is a polyhedral structure on $\Trop(X)$ such that for each maximal polyhedron $P$ in $\Trop(X \cap T)$, the closure $\overline{P}$ is a union of polyhedra, each of which is maximal in its respective stratum.
\end{itemize}
While previous results on faithful tropicalizations \cite{CHW14, DP14}, required careful study of Berkovich spaces and their skeleta, that analysis is now generalized and absorbed into the proof of the above criteria, so that faithfulness may be checked by exclusively working ``downstairs,'' with the tropicalization.

In practice, however, it can be quite difficult to check the criteria of \cite{GRW15}. For instance, the Grassmannian of planes $\Gr(2,n)$ is faithfully tropicalized by its Pl\"{u}cker embedding, as originally proved in \cite{CHW14}. A shorter proof may be obtained with the aid of the above criteria, but one still needs the combinatorial results from Sections 4 and 5 of \cite{CHW14}, which rely on the interpretation, due to Speyer and Sturmfels \cite{SS04}, of the tropicalization as the space of phylogenetic trees. By constrast, for $k \geq 3$ there is no known combinatorial interpretation of $\Trop(\Gr(k,n))$, and the question of faithfulness remains open. Moreover, there is no reason to expect that a ``natural'' embedding into a toric variety, such as the Pl\"{u}cker embedding, should yield a faithful tropicalization. For a given variety, there is no known procedure for obtaining a faithful tropicalization.

Here, we show that an arbitrary hypertoric variety $\fM_{\cA}$ is faithfully tropicalized by its embedding in the Lawrence toric variety $\fM_{\cA}$ by proving that the criteria listed above are satisfied. This is a notable result for a number of reasons. To our knowledge, this is the first application of Gubler, Rabinoff, and Werner's criteria to a class of tropicalizations for which faithfulness was previously unknown. We obtain several examples, in every even dimension, of varieties which are faithfully tropicalized by a ``natural'' embedding into a toric variety. These examples include the cotangent bundles of projective spaces and products of projective spaces, as well as many singular varieties.

Furthermore, we shall see that, in all but the most trivial case, the hypertoric variety $\fM_{\cA}$ in its Lawrence embedding does not meet all torus orbits in the expected dimension (Corollary \ref{cor:notproper}). This is in contrast to several other known examples of ``nice'' tropicalizations, including the moduli space $\overline{M}_{0,n}$ of stable rational curves \cite{GM10, T07}, some alternate compactifications of $M_{0,n}$ \cite{CHMR14}, and the space of logarithmic stable maps to a projective toric variety \cite{R15}. By \cite[Corollary 8.15]{GRW15}, a variety which meets all torus orbits in the expected dimension, or not at all, is faithfully tropicalized if it has multiplicity one everywhere. Since this result is not available to us, we must  explicitly describe a polyhedral structure on $\Trop(\fM_{\cA})$. This is provided by our first main result.

\begin{introhtvcones}
The tropicalization $\Trop(\fM_{\cA})$ of the hypertoric variety is the set-theoretic disjoint union of cones $C^{(F, \cR)}_{\cF}$ indexed by a flat $F$ of $\cA_0$, a face $\cR$ of the localization $\cA_F$, and a flag of flats $\cF$ of the restriction $\cA_0^F$. These cones satisfy
\[ \dim C^{(F, \cR)}_{\cF} = d + \ell(\cF) - \codim \cR. \]
This gives $\Trop(\fM_{\cA})$ the combinatorial structure of a finite polyhedral complex, under the closure relation
\[ C^{(F', \cR')}_{\cF'} \subseteq \overline{C^{(F, \cR)}_{\cF}} \]
if and only if the following conditions hold:
\begin{itemize}
\item $F \subseteq F'$;
\item $\cR' \subseteq \overline{\cR}$;
\item $F'$ is a flat in $\cF$, and $\trunc_{F'}(\cF)$ is a refinement of $\cF'$.
\end{itemize}
Moreover, this gives each stratum $\Trop(\fM_{\cA} \cap O(\sigma_{F, \cR}))$ the structure of a polyhedral fan, which is balanced when all cones are given weight one.
\end{introhtvcones}

Equipped with Theorem \ref{thm:htvcones}, we can describe the interplay between the fan $\Trop(\fM_{\cA} \cap \widetilde{T})$, where $\widetilde{T}$ is the dense torus of $\fB_{\cA}$, and the fan of the toric variety $\fB_{\cA}$. Each of these fans has cones described by the combinatorics of the arrangement $\cA$: cones in the fan of $\fB_{\cA}$ correspond to faces of localizations of $\cA$, while $\Trop(\fM_{\cA} \cap \widetilde{T})$ has cones indexed by flags of flats. We see that the criteria of \cite{GRW15} are satisfied, proving the tropicalization is faithful.

\begin{introfaithful}
There is a unique continuous section $s \colon \Trop(\fM_{\cA}) \to \fM_{\cA}^{\an}$ of the tropicalization map.
\end{introfaithful}

The rest of the paper is outlined as follows. We review the combinatorics of hyperplane arrangements in Section \ref{sec:arrangements}, and in Section \ref{sec:toric} we recall basic facts about analytification and tropicalization. Sections \ref{sec:ltv} and \ref{sec:htv} contain the constructions of the Lawrence toric variety $\fB_{\cA}$ and the hypertoric variety $\fM_{\cA}$, respectively. In Section 6, we study the tropicalization $\Trop(\fM_{\cA})$ and present our main results.
\\ \\
\textbf{Acknowledgements.} I would like to thank Nick Proudfoot for his guidance and many helpful suggestions. I am also grateful to Dhruv Ranganathan for his comments on an earlier draft of this paper.
\\ \\
\textbf{Notation.} Throughout, we fix a lattice $M \cong \ZZ^d$ and a field $K$, complete with respect to a non-Archimedean valuation which may be trivial. The dual lattice to $M$ is $N = \Hom(M,\ZZ)$, and we set $M_{\RR} = M \otimes_{\ZZ} \RR$ and $N_{\RR} = N \otimes_{\ZZ} \RR = \Hom(M, \RR)$.

\section{Hyperplane arrangements} \label{sec:arrangements}

Given a finite set $E$, a tuple $a \in N^E$ of nonzero elements, and $r \in \ZZ^E$, we define the corresponding \textbf{arrangement} $\cA = \cA(a,r)$ of affine hyperplanes in $M_{\RR}$ to be the multiset consisting of the hyperplanes
\[ H_e = \{ m \in M_{\RR} \mid \langle m, a_e \rangle + r_e = 0 \} \]
for $e \in E$. If $\{a_e \mid e \in E\}$ generates the lattice $N$ and each $a_e$ is a primitive vector, we say that $a$ is a \textbf{primitive spanning configuration}. We shall always consider arrangements defined by a primitive spanning configuration, except for the localizations of such arrangements (see below). The hyperplane $H_e$ is cooriented by the integral normal vector $a_e$, with ``positive'' and ``negative'' closed halfspaces,
\[ H_e^+ = \{ u \in M_{\RR} \mid \langle u, a_e \rangle + r_e \leq 0 \} \]
and
\[ H_e^- = \{ u \in M_{\RR} \mid \langle u, a_e \rangle + r_e \geq 0 \}, \]
respectively.

An arrangement $\cA = \cA(a,r)$ is \textbf{central} if $r = 0$, so that each hyperplane $H_e$ is a linear subspace of $M_{\RR}$. On the other hand, $\cA$ is \textbf{simple} if the intersection of any $k$ hyperplanes is either empty or has codimension $k$, and $\cA$ is \textbf{unimodular} if every collection of $d$ linearly independent normal vectors $\{a_{e_1}, \ldots, a_{e_d}\}$ is a basis of the lattice $N$. An arrangement is \textbf{smooth} if it is both simple and unimodular. Given $\cA = \cA(a,r)$, we call the arrangement $\cA_0 = \cA(a,0)$ the \textbf{centralization} of $\cA$. We let $(H_e)_0$ denote the hyperplane indexed by $e$ in $\cA_0$.

An arrangement $\cA$ defines a loop-free matroid on the set $E$, via the centralization $\cA_0$. This matroid is defined by the rank function $\rk S = \codim H_S$, where $H_S = \cap_{e \in S} (H_e)_0$ (equivalently, $\rk S$ is equal to the dimension of the subspace of $N_{\RR}$ spanned by $\{a_s \mid s \in S\}$). We set $\rk \cA_0 = \rk E$, which is equal to $d$ because $a$ is spanning. A subset $F \subseteq E$ is a \textbf{flat} of $\cA_0$ if it is maximal for its rank, or equivalently
\[ F = \{e \in E \mid (H_e)_0 \supseteq H_F\}. \]
By a \textbf{flag} $\cF$ of flats in $\cA_0$, we mean a chain
\[ \emptyset = F_0 \subsetneq F_1 \subsetneq \cdots \subsetneq F_{k-1} \subsetneq F_k = E \]
where each $F_i$ is a flat. The length of such a flag, denoted $\ell(\cF)$, is the number $k$ of nonempty flats in $\cF$, so that a maximal flag has length $\rk \cA$.

Given a flat $F$ of $\cA_0$, the \textbf{restriction} of $\cA_0$ to $F$, denoted $\cA_0^F$, is the arrangement of hyperplanes $\{(H_e)_0 \cap H_F \mid e \notin F\}$ in the vector space $H_F$. Flats of $\cA_0^F$ are subsets of $E \setminus F$ of the form $F' \setminus F$ for $F'$ a flat of $\cA_0$ containing $F$. We shall therefore identify flats of $\cA_0^F$ with flats of $\cA_0$ which contain $F$. In particular, a flag of flats in $\cA_0^F$ is a chain of flats in $\cA_0$ beginning at $F$ and ending at $E$.

On the other hand, given any subset $S \subseteq E$, we define the \textbf{localization} of the (possibly non-central) arrangement $\cA$ at $S$ to be the arrangement of hyperplanes $\{H_e \mid e \in S\}$ in the vector space $M_{\RR}$. Since a subset of a primitive spanning configuration need not span $N$, the localization $\cA_S$ may not be defined by a primitive spanning configuration, and therefore the intersection of all hyperplanes in $\cA_S$ may have positive dimension. Note that $(\cA_0)_S = (\cA_S)_0$, and that flats of $(\cA_S)_0$ are the flats of $\cA_0$ which are contained in $S$.

An arrangement $\cA$ assigns a \textbf{sign vector} $\sgn_{\cA}(m) \in \{+,0,-\}^E$ to each $m \in M_{\RR}$, via
\[ \sgn_{\cA}(m)_e = \begin{cases}
+ & \text{if $m \in H_e^+ \setminus H_e$}, \\
0 & \text{if $m \in H_e$}, \\
- & \text{if $m \in H_e^- \setminus H_e$}.
\end{cases} \]
A nonempty fiber of $\sgn_{\cA} \colon M_{\RR} \to \{+,0,-\}^E$ is called a \textbf{face} of the arrangement $\cA$. A face consisting of a single point is called a \textbf{vertex} of $\cA$.

Given a face $\cR$ of $\cA$, we define the sets $E^+(\cR) = \{e \in E \mid \cR \subseteq H_e^+\}$, $E^-(\cR) = \{e \in E \mid \cR \subseteq H_e^-\}$, and $E^0(\cR) = E^+(\cR) \cap E^-(\cR)$. Then for any $m \in \cR$, we have
\[ \sgn_{\cA}(m)_e = \begin{cases}
+ & \text{if $e \in E^+(\cR) \setminus E^0(\cR)$}, \\
0 & \text{if $e \in E^0(\cR)$}, \\
- & \text{if $e \in E^-(\cR) \setminus E^0(\cR)$}.
\end{cases} \]
Notice that the closure of a face is the intersection of all halfspaces which contain it:
\begin{equation} \label{eq:clface}
\overline{\cR} = \Big( \bigcap_{e \in E^+(\cR)} H_e^+ \Big) \cap \Big( \bigcap_{e \in E^-(\cR)} H_e^- \Big).
\end{equation}
The codimension of $\cR$ in $M_{\RR}$ is the codimension of the intersection of all hyperplanes containing it, so that
\begin{equation} \label{eq:codimface}
\codim \cR = \codim \bigcap_{e \in E^0(\cR)} H_e = \codim \bigcap_{e \in E^0(\cR)} (H_e)_0 = \rk (\cA_{E^0(\cR)})_0.
\end{equation}

The above discussion and notation applies to localizations of $\cA$ as well. For instance, if $S \subseteq E$, then the arrangement $\cA_S$ defines a sign function $\sgn_{\cA_S} \colon M_{\RR} \to \{+, 0, -\}^S$ which determines the faces of $\cA_S$. Given a face $\cR$ of $\cA_S$, the sets $S^+(\cR)$, $S^-(\cR)$, and $S^0(\cR)$ are defined as above, and $\overline{\cR}$ and $\codim \cR$ are computed in the same way, with $E$ replaced by $S$.

\begin{lem} \label{lem:clface}
Let $S \subseteq E$. If $\cR$ is a face of the localization $\cA_S$ and $\cR'$ is a face of $\cA$, then $\cR' \subseteq \overline{\cR}$ if and only if $E^+(\cR') \supseteq S^+(\cR)$ and $E^-(\cR') \supseteq S^-(\cR)$.
\end{lem}

\begin{proof}
Suppose $\cR' \subseteq \overline{\cR}$. Since the halfspaces $H_e^+$ and $H_e^-$ are closed, any halfspace which contains $\cR$ also contains $\overline{\cR}$, hence contains $\cR'$. That is, $E^+(\cR') \supseteq S^+(\cR)$ and $E^-(\cR') \supseteq S^-(\cR)$

Conversely, suppose $E^+(\cR') \supseteq S^+(\cR)$ and $E^-(\cR') \supseteq S^-(\cR)$. Then for each $e \in S^+(\cR)$, we also have $e \in E^+(\cR')$ and therefore $\cR' \subseteq H_e^+$. Similarly, $\cR' \subseteq H_e^-$ for each $e \in S^-(\cR)$. By \eqref{eq:clface}, we conclude that $\cR' \subseteq \overline{\cR}$.
\end{proof}

\section{Toric varieties and tropicalization} \label{sec:toric}

Let $T = \Spec K[M]$ be the split $K$-torus with character lattice $M$ and cocharacter lattice $N$. Let $\Delta$ be a (pointed) rational polyhedral fan in $N_{\RR}$.

Each cone $\sigma \in \Delta$ defines an affine toric variety $Y_{\sigma} = \Spec K[\sigma^{\vee} \cap M]$ with dense torus $T$. For $\tau \prec \sigma$ in $\Delta$, $Y_{\tau}$ is naturally an open subvariety of $Y_{\sigma}$. Gluing along these identifications, we obtain the $T$-toric variety $Y_{\Delta} = \cup_{\sigma \in \Delta} Y_{\sigma}$ defined by $\Delta$.

For each $\sigma \in \Delta$, we have the torus orbit $O(\sigma) = \Spec K[M(\sigma)]$, where $M(\sigma) = \sigma^{\perp} \cap M$, and $Y_{\Delta}$ is (set-theoretically) partitioned into torus orbits: $Y_{\Delta} = \sqcup_{\sigma \in \Delta} O(\sigma)$. The dual lattice of $M(\sigma)$ is $N(\sigma) = N/(\langle \sigma \rangle \cap N)$, where $\langle \sigma \rangle = \RR\sigma$ is the linear span of $\sigma$ in $N_{\RR}$. We write $M_{\RR}(\sigma) = M(\sigma) \otimes_{\ZZ} \RR$ and $N_{\RR}(\sigma) = N(\sigma) \otimes_{\ZZ} \RR$. For $\tau \prec \sigma$ we let $\pi^{\tau}_{\sigma} \colon N_{\RR}(\tau) \to N_{\RR}(\sigma)$ be the projection. If $\tau = 0$, then we simply write $\pi_{\sigma} \colon N_{\RR} \to N_{\RR}(\sigma)$. For a cone $\tau \in \Delta$, the orbit closure $\overline{O(\tau)}$ is a toric variety with dense torus $O(\tau)$. Its fan is the collection of cones $\{ \pi_{\tau}(\sigma) \mid \sigma \succ \tau \}$ in $N_{\RR}(\tau)$.

A morphism of tori $f \colon T \to T'$ is given by a map of lattices $f_* \colon N \to N'$, or equivalently by the dual map $f^* \colon M' \to M$. Given fans $\Delta$ in $N_{\RR}$ and $\Delta'$ in $N'_{\RR}$, the morphism $f$ extends to a map of toric varieties $Y_{\Delta} \to Y_{\Delta'}$ if for each cone $\sigma' \in \Delta$, there exists a cone $\sigma' \in \Delta'$ such that $f_*(\sigma) \subseteq \sigma'$.

Following \cite{G15}, we define a \textbf{linear subvariety} $L$ of the torus $T$ to be a subvariety in some choice of torus coordinates. That is, there exists an isomorphism $M \cong \ZZ^n$, inducing $K[M] \cong K[\ZZ^n] \cong K[x_1^{\pm 1}, \ldots, x_n^{\pm 1}]$, such that the ideal of $L$ is generated by linear forms in the $x_i$.

\begin{lem} \label{lem:linear}
If $f \colon T \to T'$ is a split surjection of tori, and $L \subseteq T'$ is a linear subvariety, then $f^{-1}(L)$ is a linear subvariety of $T$. Furthermore, $\codim_T f^{-1}(L) = \codim_{T'} L$.
\end{lem}

\begin{proof}
Let $M$ and $M'$ denote the character lattices of $T$ and $T'$, respectively. The surjection $f$ is induced by an injective map on monomials $f^* \colon M' \to M$, which is split because $f$ is. It follows that $f^*$ maps primitive elements to primitive elements. Thus, if $\{x_1, \ldots, x_n\}$ is an integral basis of $M'$, then $\{f^*(x_1), \ldots, f^*(x_n)\}$ is a linearly independent set of primitive elements in $M$, and therefore may be extended to an integral basis. 

If $L$ is generated by linear forms in the $x_i$, then pulling back along $f^*$, we see that the ideal of $f^{-1}(L)$ is generated by linear forms in the $f^*(x_i)$. Moreover, injectivity of $f^*$ guarantees that the ideal of $f^{-1}(L)$ is generated by the same number of independent linear forms as the ideal of $L$, proving that $\codim_T f^{-1}(L) = \codim_{T'} L$.
\end{proof}


Consider $\TT = \RR \cup \{ \infty \}$, a monoid under addition with the topology of a half-open interval. For a cone $\sigma \in \Delta$, we define $\overline{N}_{\RR}^{\sigma}$ to be the set of monoid homomorphisms $\Hom(\sigma^{\vee} \cap M, \TT)$. We give $\overline{N}_{\RR}^{\sigma}$ the topology of pointwise convergence. If $\tau \prec \sigma$, then $\overline{N}_{\RR}^{\tau}$ is naturally identified with the open subset of $\overline{N}_{\RR}^{\sigma}$ consisting of maps which are finite on $\tau^{\perp} \cap \sigma^{\vee} \cap M$ (in particular, $N_{\RR} =\overline{N}_{\RR}^{\{0\}}$ is an open subset of each $\overline{N}_{\RR}^{\sigma}$). Gluing along these identifications, we obtain $\overline{N}_{\RR}^{\Delta}$, a partial compactification of $N_{\RR}$. This mirrors the construction of the toric variety $Y_{\Delta}$: we have a decomposition $\overline{N}_{\RR}^{\Delta} = \cup_{\sigma \in \Delta} \overline{N}_{\RR}^{\sigma}$ analogous to the decomposition of $Y_{\Delta}$ into affine toric varieties, and a decomposition $\overline{N}_{\RR}^{\Delta} = \sqcup_{\sigma \in \Delta} N(\sigma)$ (as sets) analogous to the decomposition of $Y_{\Delta}$ into torus orbits. These two constructions are related by the process of tropicalization.


Given a $K$-variety $X$, the analytification functor assigns to $X$ a \textbf{Berkovich analytic space}, denoted $X^{\an}$ \cite[Sections 3.4 and 3.5]{B90}. This is a topological space equipped with a sheaf of analytic functions. For our purposes, it will be enough to describe the topology of $X^{\an}$ in the affine case. If $X = \Spec A$ is affine, then $X^{\an}$ is identified with the set of ring valuations $A \to \TT$ extending the valuation on $K$. This space is equipped with the coarsest topology such for every $a \in A$, the evaluation map $\ev_a \colon X^{\an} \to \TT$, $\nu \mapsto \nu(a)$ is continuous.

Let the torus $T$ and fan $\Delta$ be as above. The \textbf{tropicalization map} on the torus is the continuous surjection
\[ \trop \colon T^{\an} \to N_{\RR} \]
which takes a valuation $w \colon K[M] \to \TT$ to its restriction $w|_M \colon M \to \RR$. More generally, for a cone $\sigma \in \Delta$, we have a \textbf{tropicalization map}
\[ \trop \colon Y_{\sigma}^{\an} \to \overline{N}_{\RR}^{\sigma} = \Hom(\sigma^{\vee} \cap M, \TT) \]
which similarly takes a valuation $w \colon K[\sigma^{\vee} \cap M] \to \TT$ to its restriction to $\sigma^{\vee} \cap M$. These maps glue to give a tropicalization map $\trop \colon Y_{\Delta}^{\an} \to \overline{N}_{\RR}^{\Delta}$. This map is a continuous and proper surjection, which has the property that its restriction to each torus orbit is the usual tropicalization $O(\sigma)^{\an} \to N_{\RR}(\sigma)$ for a torus. 

Given a closed subvariety $X \subseteq Y_{\Delta}$, the \textbf{tropicalization} $\Trop(X) \subseteq \overline{N}_{\RR}^{\sigma}$ of $X$ is the image of $X^{\an} \subseteq Y^{\an}$ under $\trop$. If $X$ is a subvariety of a torus, then $\Trop(X)$ may be given the structure of a finite polyhedral complex, which is a (not necessarily pointed) fan if $X$ is defined over a subfield of $K$ having trivial valuation. Moreover, this polyhedral complex is of pure dimension $\dim X$ and is equipped with a positive integer-valued weight function, the tropical multiplicty, with respect to which the complex is balanced \cite[Theorem 3.3.5]{MS15}.

In general, when $X$ is a subvariety of a toric variety, $\Trop(X)$ may be computed orbit-by-orbit: $\Trop(X) \cap N_{\RR}(\sigma) = \Trop(X \cap O(\sigma))$. Thus, $\Trop(X)$ is a partial compactification of the balanced polyhedral complex $\Trop(X \cap T)$ by lower-dimensional finite polyhedral complexes. If $X = \overline{X \cap T}$ (in particular, if $X$ is irreducible and $X \cap T$ is nonempty), then $\Trop(X)$ is the closure of $\Trop(X \cap T)$ in $\overline{N}_{\RR}^{\sigma}$.

Tropicalization is functorial with respect to morphisms of toric varieties. Let $f \colon Y_{\Delta} \to Y_{\Delta'}$ be such a map. For $\sigma \in \Delta$, there exists $\sigma' \in \Delta'$ such that $f_*(\sigma) \subseteq \sigma'$. For such a $\sigma'$, the restriction of $f^*$ gives a map $M(\sigma') \to M(\sigma)$, inducing $\overline{N}_{\RR}^{\sigma} \to \overline{N}_{\RR}^{\sigma'}$. These maps glue to give a map $\overline{N}_{\RR}^{\Delta} \to \overline{N}_{\RR}^{\Delta'}$, denoted $\Trop(f)$. See \cite{P09} for details. By \cite[Corollary 6.2.17]{MS15}, if $X \subseteq Y_{\Delta}$, then $\Trop(f)(\Trop(X)) = \Trop(f(X))$.

\begin{ex}
Of particular importance to us will be the tropicalization of a linear space. Given a hyperplane arrangement $\cA$, we obtain a linear subspace $L_{\cA_0} \subseteq \AA^E = \Spec K[x_e \mid e \in E]$, which depends only on the centralization $\cA_0$. For each relation $\sum_{e \in E} c_e a_e = 0$ in $N$, we have the corresponding linear form $\sum_{e \in E} c_e x_e$, and $L_{\cA_0}$ is cut out by the ideal generated by these forms. Note that $\dim L_{\cA_0} = \rk \cA_0$. Every linear subspace of $\AA^E$ which is not contained in a coordinate subspace arises from a central arrangement in this way (the condition that $L_{\cA_0}$ not lie in a coordinate hyperplane is equivalent to the matroid of $\cA_0$ being loop-free).

The torus orbits of $\AA^E$ are indexed by subsets $S \subseteq E$, where $S$ corresponds to the torus $\GG_m^{E \setminus S}$ defined by $x_e = 0$ if and only if $e \in S$. We note that $L_{\cA_0} \cap \GG_m^{E \setminus S}$ is nonempty if and only if $S$ is a flat of $\cA_0$, in which case it is the intersection of $L_{\cA_0^S} \subseteq \AA^{E \setminus S}$ with $\GG_m^{E \setminus S}$.

Given a flat $F$ of $\cA_0$, we define $\delta_{F} = \sum_{e \in F} \delta_e \in \RR^E \subseteq \TT^E$, where $\delta_e \in \RR^E$ is the basis vector corresponding to $e \in E$. For a flag $\cF$ of flats
\[ \emptyset = F_0 \subset F_1 \subset \cdots \subset F_{k-1} \subset F_k = E, \]
we have the $\ell(k)$-dimensional cone
\[ \beta_{\cF} = \RR_{\geq 0}\langle \delta_{F_1}, \cdots, \delta_{F_{k-1}}, \pm \delta_E \rangle \subseteq \RR^E. \]
The collection of cones $\beta_{\cF}$, for $\cF$ a flag of flats of $\cA_0$, defines a fan in $\RR^E$, called the \textbf{Bergman fan} of $\cA_0$ (with the \emph{fine} fan structure of \cite{AK06}). The Bergman fan of $\cA_0$ is a pure polyhedral fan of dimension $\rk \cA_0$. We have $\beta_{\cF} \prec \beta_{\cF'}$ if and only if $\cF'$ is a refinement of $\cF$. The Bergman fan is balanced with respect to the weight function which assigns each maximal cone weight one, and its support is $\Trop(L_{\cA_0} \cap \GG_m^E)$. 

Note that every cone $\beta_{\cF}$ contains the diagonal copy of $\RR$, spanned by $\delta_E$. Many authors define the Bergman fan to be the image of $\Trop(L_{\cA_0} \cap \GG_m^E)$ in the quotient of $\RR^E$ by the diagonal. We shall not adopt this convention.

The linear space $L_{\cA_0}$ intersects the torus orbit $\GG_m^{E \setminus S}$ if and only if $S$ is a flat of $\cA$, in which case its tropicalization $\Trop(L_{\cA} \cap \GG_m^{E \setminus S})$ is identified with the Bergman fan of the restriction $\cA_0^S$. Its cones are denoted $\beta^{(S)}_{\cF}$, for $\cF$ a flag of flats in $\cA_0^S$. The full tropicalization $\Trop(L_{\cA})$, together with the fan structures on its individual strata, is the \textbf{extended Bergman fan} of $\cA$.
\end{ex}

\section{Lawrence toric varieties} \label{sec:ltv}

Let $a$ be a primitive spanning configuration, and let $\cA = \cA(a,r)$ be an arrangement in $M_{\RR}$. We continue to let $T = \Spec K[M]$ be the torus with character lattice $M$. The configuration $a$ defines a surjection $\ZZ^E \to N$ taking the coordinate vector $\delta_e$ to $a_e$. Let $\Lambda$ be the kernel of this map.

Consider the antidiagonal embedding $\nabla \colon \ZZ^E \to \ZZ^E \oplus \ZZ^E$. If we denote by $\delta_e^+$ the generators of the first copy of $\ZZ^E$ and by $\delta_e^-$ the generators of the second copy of $\ZZ^E$, then $\nabla(\delta_e) = \delta_e^+ - \delta_e^-$. We obtain a commutative diagram
\begin{equation} \label{eq:sescochar}
\begin{tikzcd}
	0
		\ar{r} &
	\Lambda
		\ar{d}{\id}
		\ar{r}{\iota} &
	\ZZ^E
		\ar{d}{\nabla}
		\ar{r}{a} &
	N
		\ar{d}
		\ar{r} &
	0 \\
	0
		\ar{r} &
	\Lambda
		\ar{r}{\iota_{\nabla}} &
	\ZZ^E \oplus \ZZ^E
		\ar{r} &
	\widetilde{N}
		\ar{r} &
	0
\end{tikzcd}
\end{equation}
where the rows are exact. Here, $\widetilde{N}$ is the cokernel of $\iota_{\nabla} = \nabla \circ \iota$. Let $\rho_e^+$ and $\rho_e^-$ be the images in $\widetilde{N}$ of the generators $\delta_e^+$ and $\delta_e^-$, respectively. Then $\widetilde{N}$ is a free abelian group of rank $|E| + d$, generated by the elements $\rho_e^{\pm}$. For each relation $\sum_{e \in E} c_e a_e = 0$ in $N$, we have the relation
\[ \sum_{e \in E} c_e (\rho_e^+ - \rho_e^-) = 0 \]
in $\widetilde{N}$, and these are all of the relations satisfied by the elements $\rho_e^{\pm}$.

Let $\widetilde{M}$ be the dual lattice to $\widetilde{N}$. By dualizing \eqref{eq:sescochar}, we have that $\widetilde{M}$ is a rank $d + |E|$ sublattice of $\ZZ \langle x_e^{\pm} \mid e \in E\rangle$, where $\{x_e^{\pm}\}$ is the basis dual to $\{\delta_e^{\pm}\}$. The torus $\widetilde{T} = \Spec K[\widetilde{M}]$ is the \textbf{Lawrence torus}. The arrangement $\cA$ defines a $\widetilde{T}$-toric variety $\fB_{\cA}$, called the \textbf{Lawrence toric variety} of $\cA$. In the literature, $\fB_{\cA}$ is traditionally identified as the GIT quotient of $T^*\AA^E \cong \AA^E \times \AA^E$ by $G = \Spec K[\Lambda^*]$ with respect to the character $\alpha = \iota^*(r) \in \Lambda^*$.

For our purposes, we will need an explicit description of the fan $\Delta_{\cA}$ of $\fB_{\cA}$. For any subset $S \subseteq E$ and face $\cR$ of the localization $\cA_S$, we define $\sigma_{S, \cR}$ to be the cone in $\widetilde{N}_{\RR}$ whose rays are generated by the integral vectors
\[ \{ \rho_e^+ \mid e \in S^+(\cR)\} \cup \{ \rho_f^- \mid f \in S^-(\cR)\}. \]
Hausel and Sturmfels \cite[Proposition 4.3]{HS02} proved that the cones $\sigma_{E, \xi}$, for $\xi$ a vertex of $\cA$, are precisely the maximal cones of $\Delta_{\cA}$, and hence $\Delta_{\cA}$ consists of these cones together with all of their faces. We show that these faces are precisely the remaining cones $\sigma_{S, \cR}$.

\begin{prop} \label{prop:cones}
The Lawrence fan is the set of cones
\[ \Delta_{\cA} = \{ \sigma_{S, \cR} \mid \text{$S \subseteq E$ and $\cR$ is a face of $\cA_S$} \}, \]
with face relations
\[ \sigma_{S', \cR'} \prec \sigma_{S, \cR} \quad \text{if and only if} \quad \text{$S \supseteq S'$ and $\cR \subseteq \overline{\cR'}$.} \]
Moreover, $\dim \sigma_{S, \cR} = |S| + \codim \cR$.
\end{prop}

\begin{proof}
Suppose $\tau \prec \sigma_{S, \cR}$ is a face for some $S \subseteq E$ and face $\cR$ of $\cA_S$. Then $\tau = u^{\perp} \cap \sigma_{S, \cR}$ for some $u \in \sigma_{S, \cR}^{\vee} \subseteq \widetilde{M}_{\RR}$. We also know that $\tau$ is generated by a subset of the rays of $\sigma_{S, \cR}$. That is, $\tau$ is the cone generated by
\[ \{\rho_i^+ \mid i \in A \} \cup \{\rho_j^- \mid j \in B \} \]
for some subsets $A \subseteq S^+(\cR)$ and $B \subseteq S^-(\cR)$. 

Fix any point $p \in \cR$, and set
\[ m_{\epsilon} = \nabla^*(\epsilon u) + p \in M_{\RR}, \]
for $\epsilon > 0$. For a sufficiently small choice of $\epsilon$, we have

\[ \sgn_{\cA_{A \cup B}}(m_{\epsilon})_e = \begin{cases}
+ & \text{if $e \in A \setminus (A \cap B)$}, \\
0 & \text{if $e \in A \cap B$}, \\
- & \text{if $e \in B \setminus (A \cap B)$}.
\end{cases} \]
In other words, $m_{\epsilon}$ lies in a face $\cR'$ of $\cA_{A \cup B}$ satisfying $(A \cup B)^+(\cR') = A$ and $(A \cup B)^-(\cR') = B$. This shows that $\tau = \sigma_{A \cup B, \cR'}$. Since $A \subseteq S^+(\cR)$ and $B \subseteq S^-(\cR)$, we have $\cR \subseteq \overline{\cR'}$ by Lemma \ref{lem:clface}.

Conversely, let $S' \subseteq S$ and let $\cR'$ be a face of $\cA_{S'}$ with $\cR \subseteq \overline{\cR'}$. We shall show that $\sigma_{S', \cR'}$ is a face of $\sigma_{S, \cR}$.

By Lemma \ref{lem:clface}, we have $S^+(\cR) \supseteq S'^+(\cR')$ and $S^-(\cR) \supseteq S'^-(\cR')$. Fix some $p \in \cR$ and $m \in \cR'$, and set
\[ u = \sum_{i \in S'^{-}(\cR')} \langle m - p, a_i \rangle x_i^+ - \sum_{j \in S'^{+}(\cR') \setminus S'^0(\cR')} \langle m - p, a_j \rangle x_j^- + \sum_{k \in S \setminus S'} (c_kx_k^+ - d_kx_k^-), \]
where $c_k$ and $d_k$ are positive real numbers chosen so that $c_k - d_k = \langle m - p, a_k \rangle$ for every $k \in S \setminus S'$. A priori, $u \in \RR \langle x_e^{\pm} \rangle$. However, it is easy to verify that
\[ \langle \nabla^* (u), \delta_e \rangle = \langle u, \nabla(\delta_e) \rangle = \langle u, \delta_e^+ - \delta_e^- \rangle \]
is equal to $\langle m - p, a_e \rangle$ for all $e \in S$. Since $\{a_e, e \in E\}$ spans $N_{\RR}$, it follows that $\nabla^*(u) = m - p \in M_{\RR}$, and therefore $u \in \widetilde{M}_{\RR}$. By design, we have $u \in \sigma_{S, \cR}^{\vee}$ and $u^{\perp} \cap \sigma_{S, \cR} = \sigma_{S', \cR'}$, proving that $\sigma_{S', \cR'} \prec \sigma_{S, \cR}$.


We now calculate the dimension of $\sigma_{S, \cR}$, which is equal to the dimension of the real vector space $\langle \sigma_{S, \cR} \rangle \subseteq \widetilde{N}_{\RR}$ spanned by it. Define
\[ V_1 = \RR \langle \rho_i^+, \rho_j^- \mid i \in S^+(\cR) \setminus S^0(\cR), j \in S^-(\cR) \setminus S^0(\cR) \rangle \]
and
\[ V_2 = \RR \langle \rho_i^+, \rho_j^- \mid i, j \in S^0(\cR) \rangle. \]
Then we clearly have $\langle \sigma_{S, \cR} \rangle = V_1 + V_2$.

Since every relation among the elements $\rho_e^{\pm}$ is of the form
\[ \sum_{e \in E} c_e(\rho_e^+ - \rho_e^-) = 0, \]
any nontrivial linear dependence among the generators of $\sigma_{S, \cR}$ must occur among the vectors $\{ \rho_i^+, \rho_j^- \mid i, j \in S^0(\cR) \}$. It follows that $\dim V_1 = |S \setminus S^0(\cR)|$ and $\langle \sigma_{S, \cR} \rangle = V_1 \oplus V_2$.

By \eqref{eq:codimface}, we have $\codim \cR = \rk (\cA_{S^0(\cR)})_0$. Choose any basis $Q \subseteq S^0(\cR)$ of the matroid of the arrangement $(\cA_{S^0(\cR)})_0$. We claim that
\begin{equation} \label{eq:V2basis}
\{ \rho_i^+, \rho_j^- \mid i \in Q, j \in S^0(\cR) \}
\end{equation}
is a basis for $V_2$. Indeed, any nontrivial linear dependence among these generators must occur among the subset $\{ \rho_i^+, \rho_j^- \mid i, j \in Q \}$, but any such dependence must be trivial because $Q$ is independent. Thus, the set in \eqref{eq:V2basis} is linearly independent. On the other hand, for any $i \in S^0(\cR) \setminus Q$, we can write $a_i = \sum_{q \in Q} d_q a_q$ for some $d_q \in \ZZ$. This implies
\[ \rho_i^+ = \rho_i^- + \sum_{q \in Q} d_q(\rho_q^+ - \rho_q^-), \]
and therefore the set in \eqref{eq:V2basis} spans $V_2$. We conclude that
\[ \dim V_2 = |Q| + |S^0(\cR)| = \codim \cR + |S^0(\cR)|, \]
and hence
\[ \dim \sigma_{S, \cR} = \dim V_1 + \dim V_2 = \codim \cR + |S|. \]

Finally, we observe that $\sigma_{S, \cR}$ has maximal dimension $d + |E|$ if and only if $S = E$ and $\cR$ is a vertex of $\cA$. Thus, we have shown that the fan consisting of these maximal cones together with all of their faces is
\[ \{\sigma_{S, \cR} \mid \text{$S \subseteq E$ and $\cR$ is a face of $\cA_S$}\}. \]
By \cite{HS02}, this fan is $\Delta_{\cA}$.
\end{proof}

\begin{rem}
In \cite{HS02}, the maximal cones of $\Delta_{\cA}$ are indexed (non-uniquely, unless $\cA$ is simple) by bases of the Gale dual of $\cA_0$. Given such a basis $\cB \subseteq E$, we obtain a vertex $\xi$ of $\cA$ by intersecting the hyperplanes $H_e$ for $e$ in the dual basis $E \setminus \cB$. The maximal cone Hausel and Sturmfels label with the basis $\cB$ is thus identified with our cone $\sigma_{E, \xi}$.

\end{rem}





\section{Hypertoric varieties} \label{sec:htv}



Consider the surjection $\widetilde{N} \to \ZZ^E$, given by $\rho_e^{\pm} \mapsto \delta_e$ (this is the cokernel of the antidiagonal embedding $N \to \widetilde{N}$ in \eqref{eq:sescochar}). Tensoring with $\RR$, we obtain a linear map $\widetilde{N}_{\RR} \to \RR^E$. Under this surjection, $\sigma_{S, \cR}$ is mapped onto the cone $\RR^S_{\geq 0}$. We thus obtain a surjective map of toric varieties $\Phi \colon \fB_{\cA} \to \AA^E$.

We define the \textbf{hypertoric variety} of the arrangement $\cA$, denoted $\fM_{\cA}$, to be the preimage of the linear space $L_{\cA_0}$ under the map $\Phi$. It is irreducible of dimension $2d$.

\begin{rem}
If $K = \CC$, the complex points of $\AA^E/L_{\cA_0}$ can be identified with the dual Lie algebra of the torus $G$, and the composition of $\Phi$ with the projection $\AA^E \to \AA^E/L_{\cA_0}$ is then the moment map $\mu$ for the hamiltonian action of $G$ on $T^*\AA^E \cong \AA^E \times \AA^E$. This endows $\fM_{\cA} = \mu^{-1}(0) \sslash_{\alpha} G$ with a canonical Poisson structure. Arbo and Proudfoot have proved that this Poisson structure makes $\fM_{\cA}$ a symplectic variety in the sense of Beauville \cite[Proposition 4.14]{AP16}. The torus $T$ acts on $\fM_{\cA}$ via its antidiagonal embedding in the Lawrence torus $\widetilde{T}$. This action is hamiltonian with moment map $\Phi|_{\fM_{\cA}}$.
\end{rem}

By \cite[Theorems 3.2 \& 3.3]{BD00}, the hypertoric variety $\fM_{\cA}$ has at worst orbifold singularities if and only if the arrangement $\cA$ is simple, and is smooth if $\cA$ is smooth. The variety $\fM_{\cA_0}$ is affine, and if $\cA$ is simple then $\fM_{\cA} \to \fM_{\cA_0}$ is an orbifold resolution of singularities.

If $\cA$ is the arrangement of coordinate hyperplanes in $M_{\RR}$, then the associated hypertoric variety is $\fM_{\cA} \cong T^*\AA^d \cong \AA^{2d}$. The cotangent bundles of products of projective spaces may also be realized as hypertoric varieties. The polytope of $\PP^d$ is a $d$-simplex in $M_{\RR}$ cut out by $d+1$ affine hyperplanes. The hypertoric variety associated to the arrangement consisting of these hyperplanes is isomorphic to $T^*\PP^d$. The same procedure realizes the cotangent bundle of a product of projective spaces as a hypertoric variety. In general, if $Y$ is a projective toric variety with polyhedron $P$, then the arrangement of hyperplanes cutting out $P$ defines a hypertoric variety which contains $T^*Y$ as a dense open subset \cite[Theorem 7.1]{BD00}.

The affine space $\AA^E$ has torus orbits $\GG_m^{E \setminus S}$, indexed by the set $S \subseteq E$ of coordinates which vanish. The preimage $\Phi^{-1}(\GG_m^{E \setminus S})$ is the disjoint union (as sets) of all $\widetilde{T}$-orbits $O(\sigma_{S, \cR})$ for $\cR$ a face of $\cA_S$. The restriction of $\Phi$ to an orbit $O(\sigma_{S, \cR})$ is a surjection onto $\GG_m^{E \setminus S}$, which is split because $\widetilde{T} \to \GG_m^E$ is split.

\begin{prop} \label{prop:linear}
Let $S \subseteq E$ be a subset and let $\cR$ be a face of $\cA_S$. The intersection $\fM_{\cA} \cap O(\sigma_{S, \cR})$ is nonempty if and only if $S$ is a flat of $\cA_0$, in which case it is a linear subvariety of $O(\sigma_{S, \cR})$ of dimension $2d - \rk S - \codim \cR$. In particular, $\fM_{\cA} \cap O(\sigma_{S, \cR})$ is irreducible.
\end{prop}

\begin{proof}
Since $\fM_{\cA} \cap O(\sigma_{S, \cR})$ is the preimage of $L_{\cA_0} \cap \GG_m^{E \setminus S}$ under the surjection $\Phi|_{O(\sigma_{S, \cR})} \colon O(\sigma_{S, \cR}) \to \GG_m^{E \setminus S}$, we have that $\fM_{\cA} \cap O(\sigma_{S, \cR}) \neq \emptyset$ if and only if $L_{\cA_0} \cap \GG_m^{E \setminus S} \neq \emptyset$, and this occurs if and only if $S$ is a flat of $\cA_0$.

Suppose now that $S$ is a flat of $\cA_0$. Then $L_{\cA_0} \cap \GG_m^{E \setminus S}$ is a linear subvariety of $\GG_m^{E \setminus S}$ and $\Phi|_{O(\sigma_{S, \cR})}$ is a split surjection. By Lemma \ref{lem:linear}, $\fM_{\cA} \cap O(\sigma_{S, \cR})$ is a linear subvariety of $O(\sigma_{S, \cR})$ of codimension equal to the codimension of $L_{\cA_0} \cap \GG_m^{E \setminus S}$ in $\GG_m^{E \setminus S}$. Since $\dim (L_{\cA_0} \cap \GG_m^{E \setminus S}) = \rk \cA_0^S = d - \rk S$, we have
\[ \codim_{O(\sigma_{S, \cR})} (\fM_{\cA} \cap O(\sigma_{S, \cR})) = |E \setminus S| - d + \rk S. \]
By Proposition \ref{prop:cones},
\begin{align*}
\dim (\fM_{\cA} \cap O(\sigma_{S, \cR}))
	&= \dim O(\sigma_{S, \cR}) - (|E \setminus S| - d + \rk S) \\
	&= (d - \codim \cR + |E \setminus S|) - (|E \setminus S| - d + \rk S) \\
	&= 2d - \rk S - \codim \cR.
\end{align*}
\end{proof}

In general, if $X$ is a subvariety or a toric variety, then its intersection with a torus orbit corresponding to a cone $\sigma$ has expected dimension $\dim X - \dim \sigma$. If the intersection has the expected dimension, we say that $X$ intersects the torus orbit \textbf{properly}. By \cite[Corollary 8.15]{GRW15}, a subvariety which intersects each torus orbit either properly or not at all is automatically faithfully tropicalized. We now see that, except in trivial cases, hypertoric varieties \emph{do not} meet all torus orbits properly.

\begin{cor} \label{cor:notproper}
The hypertoric variety $\fM_{\cA}$ does not meet each torus orbit of $\fB_{\cA}$ properly unless $\fM_{\cA} \cong \AA^{2d}$.
\end{cor}

\begin{proof}
Let $F$ be a flat of $\cA_0$, so that $\fM_{\cA} \cap O(\sigma_{F, \cR})$ is nonempty. By Proposition \ref{prop:cones}, the expected dimension of this intersection is
\[ \dim \fM_{\cA} - \dim \sigma_{F, \cR} = 2d - |F| - \codim \cR. \]
By Proposition \ref{prop:linear}, $\dim (\fM_{\cA} \cap O(\sigma_{S, \cR}))$ agrees with the expected dimension if and only if $|F| = \rk F$. This can only occur for every flat $F$ if the matroid of $\cA_0$ is uniform of rank $d = |E|$. This, in turn, occurs if and only if $\cA$ is (a translate of) the coordinate hyperplanes in $\RR^d$. In this case, $\fM_{\cA} \cong \AA^{2d}$.
\end{proof}

\section{Tropicalization of the hypertoric variety} \label{sec:trophtv}

We now describe the structure of $\Trop(\fM_{\cA})$, the tropicalization of the hypertoric variety $\fM_{\cA}$ induced by its canonical embedding into $\fB_{\cA}$. By functoriality of tropicalization, $\Trop(\Phi)$ gives a surjection $\Trop(\fM_{\cA}) \to \Trop(L_{\cA_0})$. Given a flat $F$ of $\cA_0$, the stratum $\Trop(L_{\cA_0} \cap \GG_m^{E \setminus F})$ of $\Trop(L_{\cA_0})$ is the Bergman fan of the restriction $\cA_0^F$. This has a polyhedral fan structure with cones $\beta^{(F)}_{\cF}$ indexed by flags $\cF$ of flats of $\cA_0^F$. Given such a flag $\cF$, let $C_{\cF}^{(F, \cR)}$ be the preimage of $\beta^{(F)}_{\cF}$ under the surjection $\Trop(\fM_{\cA} \cap O(\sigma_{F, \cR})) \to \Trop(L_{\cA_0} \cap \GG_m^{E \setminus F})$. Since every Bergman fan is balanced when every maximal cone is given weight one, each stratum $\Trop(\fM_{\cA} \cap O(\sigma_{F, \cR}))$ inherits this structure of a balanced polyhedral fan with cones $C_{\cF}^{(F, \cR)}$ and all weights equal to one. Our main theorem describes how these fans are pieced together.

\begin{thm} \label{thm:htvcones}
The tropicalization $\Trop(\fM_{\cA})$ of the hypertoric variety is the set-theoretic disjoint union of cones $C^{(F, \cR)}_{\cF}$ indexed by a flat $F$ of $\cA_0$, a face $\cR$ of the localization $\cA_F$, and a flag of flats $\cF$ of the restriction $\cA_0^F$. These cones satisfy
\[ \dim C^{(F, \cR)}_{\cF} = d + \ell(\cF) - \codim \cR. \]
This gives $\Trop(\fM_{\cA})$ the combinatorial structure of a finite polyhedral complex, under the closure relation
\begin{equation} \label{eq:htvcones}
C^{(F', \cR')}_{\cF'} \subseteq \overline{C^{(F, \cR)}_{\cF}}
\end{equation}
if and only if the following conditions hold:
\begin{itemize}
\item $F \subseteq F'$;
\item $\cR' \subseteq \overline{\cR}$;
\item $F'$ is a flat in $\cF$, and $\trunc_{F'}(\cF)$ is a refinement of $\cF'$.
\end{itemize}
Moreover, this gives each stratum $\Trop(\fM_{\cA} \cap O(\sigma_{F, \cR}))$ the structure of a polyhedral fan, which is balanced when all cones are given weight one.
\end{thm}

Given a flat $F$ and a face $\cR$ of $\cA_F$, there are two fans which live in $\Trop(O(\sigma_{F, \cR})) = N_{\RR}(\sigma_{F, \cR})$: the fan $\Trop(\fM_{\cA} \cap O(\sigma_{F, \cR}))$ and the fan of the orbit closure $\overline{O(\sigma_{F, \cR})}$. The former fan has cones $C^{(F, \cR)}_{\cF}$ indexed by flags of flats in $\cA_0^F$, while the latter consists of the projections of the cones $\sigma_{S, \cR'}$ with $\sigma_{S, \cR'} \succ \sigma_{F, \cR}$ (by Proposition \ref{prop:cones}, this is equivalent to $S \supseteq F$ and $\cR' \subseteq \overline{\cR}$). The following lemma relates these two fans.

\begin{lem} \label{lem:clcone}
Let $F$ be a flat of $\cA_0$ and $\cR$ a face of $\cA_F$. Given a set $S \subseteq E$ which contains $F$ and a face $\cR'$ of $\cA_S$ contained in $\overline{\cR}$, the intersection
\[ C_{\cF}^{(F, \cR)} \cap \relint (\pi_{\sigma_{F, \cR}}(\sigma_{S, \cR'})), \]
for $\cF$ a flag of flats of $\cA_0^F$, is nonempty if and only if $S$ is a flat of $\cA$ which appears in the flag $\cF$. In this case,
\[ \overline{C^{(F, \cR)}_{\cF}} \cap \Trop(O(\sigma_{S, \cR'})) = C^{(F, \cR)}_{\trunc_S(\cF)}. \]
\end{lem}

\begin{proof}
First, because $\Phi(\sigma_{F, \cR}) = \RR^{F}_{\geq 0}$, we have that
\[ \Trop(\Phi)|_{\Trop(O(\sigma_{F, \cR}))} \colon \Trop(O(\sigma_{F, \cR})) = \widetilde{N}_{\RR}(\sigma_{F, \cR}) \to \RR^{E \setminus F} = \RR^E/\RR^F \]
is given by $[v] \mapsto [\Phi(v)]$.

Now, suppose $v \in C_{\cF}^{(F, \cR)} \cap \relint (\pi_{\sigma_{F, \cR}}(\sigma_{S, \cR'}))$. By Proposition \ref{prop:cones}, every vector in $\sigma_{S, \cR'}$ as a linear combination with positive coefficients of the generators $\rho_e^+, \rho_f^-$ for $e \in S^+(\cR')$, $f \in S^-(\cR')$. Then $v$ is the image of such a vector under $\pi_{\sigma_{F, \cR}}$, which kills all generators of $\sigma_{S, \cR'}$ indexed by elements of $F$ (since $F^{\pm}(\cR) \subseteq S^{\pm}(\cR')$ by Lemma \ref{lem:clface}). Since the square
\[
\begin{tikzcd}
	\widetilde{N}_{\RR}
		\ar{r}{\Trop(\Phi)}
		\ar{d}[swap]{\pi_{\sigma_{F, \cR}}} &
	\RR^E
		\ar{d} & \\
	\widetilde{N}_{\RR}(\sigma_{F, \cR})
		\ar{r}{\Trop(\Phi)} &
	\RR^{E \setminus F}
\end{tikzcd}
\]
commutes, it follows that $\Trop(\Phi)(v) \in \RR^{E \setminus F}$ will lie in $\RR^{S \setminus F}_{>0} \cap \beta^{(F)}_{\cF}$. By the definition of $\beta^{(F)}_{\cF}$, this intersection is nonempty if and only if $S$ is a flat in $\cF$.

Conversely, suppose $S$ is a flat in $\cF$. For $e \in S \setminus F$, define
\[ \gamma_e = \begin{cases}
\frac{1}{2}(\rho_e^+ + \rho_e^-) & \text{if $e \in S^0(\cR')$,} \\
\rho_e^+ & \text{if $e \in S^+(\cR) \setminus S^0(\cR')$,} \\
\rho_e^- & \text{if $e \in S^-(\cR) \setminus S^0(\cR')$,}
\end{cases} \]
and let $v = \sum_{e \in S \setminus F} \gamma_e$. Then $v \in \relint (\pi_{\sigma_{F, \cR}}(\sigma_{S, \cR'}))$ by design. Since $\Trop(\Phi)(v) = \delta_{S \setminus F} \in \beta_{\cF}$, it follows that $v \in \Trop(\Phi)^{-1}(\beta_{\cF}) = C_{\cF}^{(F, \cR)}$, and therefore $C_{\cF}^{(F, \cR)} \cap \relint (\pi_{\sigma_{F, \cR}}(\sigma_{S, \cR'})) \neq \emptyset$.

Supposing now that $S$ is a flat in $\cF$, we have
\[ \overline{C^{(F, \cR)}_{\cF}} \cap \Trop(O(\sigma_{S, \cR'})) = \pi^{\sigma_{F, \cR}}_{\sigma_{S, \cR'}}(C^{(F, \cR)}_{\cF}) \]
by \cite[Lemma 3.9]{OR13}, so we need only prove that this projection coincides with $C^{(S, \cR')}_{\trunc_S(\cF)}$. By commutativity of the square
\[
\begin{tikzcd}
	\widetilde{N}_{\RR}(\sigma_{F, \cR})
		\ar{r}{\Trop(\Phi)}
		\ar{d}[swap]{\pi^{\sigma_{F, \cR}}_{\sigma_{S, \cR'}}} &
	\RR^{E \setminus F}
		\ar{d} & \\
	\widetilde{N}_{\RR}(\sigma_{S, \cR'})
		\ar{r}{\Trop(\Phi)} &
	\RR^{E \setminus S}
\end{tikzcd}
\]
we see that $\Trop(\Phi)$ maps $\pi^{\sigma_{F, \cR}}_{\sigma_{S, \cR'}}(C^{(F, \cR)}_{\cF})$ onto $\beta^{(S)}_{\trunc_S(\cF)}$, which shows that $\pi^{\sigma_{F, \cR}}_{\sigma_{S, \cR'}}(C^{(F, \cR)}_{\cF}) \subseteq C^{(S, \cR')}_{\trunc_S(\cF)}$. 

On the other hand, if $w \in C^{(S, \cR')}_{\trunc_S(\cF)}$ and $v$ is a preimage of $w$ under $\pi^{\sigma_{F, \cR}}_{\sigma_{S, \cR'}}$, then $\Trop(\Phi)(v)$ need not lie in $\beta^{(F)}_{\cF}$, so that $v$ need not be in $C^{(F, \cR)}_{\cF}$. However, we can choose $x \in \RR^{S \setminus F}$ so that $\Trop(\Phi)(v) + x \in \beta_{\cF}$. Let $v' \in \widetilde{N}_{\RR}(\sigma_{F, \cR})$ be any preimage of $x$ under $\Trop(\Phi)$. Then, since $\pi^{\sigma_{F, \cR}}_{\sigma_{S, \cR'}}(v') = 0$, we have that $v + v'$ is a preimage of $w$ in $C^{(F, \cR)}_{\cF}$. This shows the reverse inclusion $C^{(S, \cR')}_{\trunc_S(\cF)} \subseteq \pi^{\sigma_{F, \cR}}_{\sigma_{S, \cR'}}(C^{(F, \cR)}_{\cF})$.
\end{proof}

We are now ready to prove the main theorem.

\begin{proof}[Proof of Theorem \ref{thm:htvcones}]
In order to have $C^{(F', \cR')}_{\cF'} \subseteq \overline{C^{(F, \cR)}_{\cF}}$, it is necessary for $\overline{C^{(F, \cR)}_{\cF}} \cap \Trop(O(\sigma_{F', \cR'}))$ to be nonempty. By Lemma \ref{lem:clcone}, this occurs precisely when $F' \supseteq F$ is a flat in $\cF$ and $\cR' \subseteq \overline{\cR}$. In this case,
\[ \overline{C^{(F, \cR)}_{\cF}} \cap \Trop(O(\sigma_{F', \cR'})) = C^{(F', \cR')}_{\trunc_F{\cF}} \]
will contain $C^{(F', \cR')}_{\cF'}$ as a face if and only if $\trunc_{F'}(\cF)$ is a refinement of $\cF'$.

A cone $C^{(F, \cR)}_{\cF}$ of $\Trop(\fM_{\cA} \cap O(\sigma_{F, \cR}))$ is maximal if $\cF$ is a maximal flag of flats. Such a cone has dimension $\dim (\fM_{\cA} \cap O(\sigma_{F, \cR}))$. By Proposition \ref{prop:linear}, this dimension equals $2d - \rk F - \codim \cR$. In general, the difference in length between a maximal flag and $\cF$ is $\rk \cA_0^F - \ell(\cF) = d - \rk F - \ell(\cF)$, and this number is also the difference in dimension between $C^{(F, \cR)}_{\cF}$ and a maximal cone. Thus, we have
\[ \dim C^{(F, \cR)}_{\cF} = 2d - \rk F - \codim \cR - (d - \rk F - \ell(\cF)) = d + \ell(\cF) + \codim \cR. \]
\end{proof}

Lemma \ref{lem:clcone}, applied to the cones $C_{\cF}^{\emptyset, \widetilde{M}_{\RR}}$ of $\Trop(\fM_{\cA} \cap \widetilde{T})$, also allows us to prove that $\fM_{\cA}$ is faithfully tropicalized by its embedding in $\fB_{\cA}$ by using the criteria of \cite{GRW15}.




\begin{thm} \label{thm:faithful}
There is a unique continuous section $s \colon \Trop(\fM_{\cA}) \to \fM_{\cA}^{\an}$ of the tropicalization map.
\end{thm}

\begin{proof}
Let $C^{\emptyset, \widetilde{M}_{\RR}}_{\cF}$ be a maximal cone in $\Trop(\fM_{\cA} \cap \widetilde{T})$. Then $C^{\emptyset, \widetilde{M}_{\RR}}_{\cF}$ has dimension $2d$ and $\cF$ is a maximal flag of flats of $\cA_0$. By Lemma \ref{lem:clcone}, for $S \subseteq E$ and $\cR$ a face of $\cA_S$, we have $C^{\emptyset, \widetilde{M}_{\RR}}_{\cF} \cap \relint(\sigma_{S, \cR}) \neq \emptyset$ if and only if $S$ is a flat in $\cF$. In this case, $\pi_{\sigma_{S,\cR}}(C^{\emptyset, \widetilde{M}_{\RR}}_{\cF}) = C^{S, \cR}_{\trunc_S(\cF)}$ is a maximal cone in $\Trop(O(\sigma_{S, \cR}))$ because $\trunc_S(\cF)$ is a maximal flag of flats of $\cA_0^F$

Since $\fM_{\cA}$ is irreducible, $\fM_{\cA} \cap \widetilde{T}$ is dense in $\fM_{\cA}$. By Proposition \ref{prop:linear}, the intersection of $\fM_{\cA}$ with any torus orbit in $\fB_{\cA}$ is either empty or equidimensional. Because $\Trop(\fM_{\cA})$ has multiplicity one everywhere, the section $s$ exists and is continuous by \cite[Theorem 8.14]{GRW15}, and is unique by \cite[Proposition 8.8]{GRW15}.
\end{proof}

Finally, we note that there is a more general notion of hypertoric variety than we have discussed here. Arbo and Proudfoot \cite{AP16} have recently shown how to construct a hypertoric variety from a zonotopal tiling $\cT$. Such a hypertoric variety is also embedded in a (generalized) Lawrence toric variety, and agrees with the variety constructed in Section \ref{sec:htv} in the case where $\cT$ is a regular tiling and hence normal to some affine arrangement. We have chosen to restrict our attention to hypertoric varieties defined by arrangements, as opposed to zonotopal tilings, primarily as a matter of convenience.

However, the Lawrence embedding of such a generalized hypertoric variety is Zariski-locally isomorphic to the embedding of $\fM_{\cA}$ into $\fB_{\cA}$, for some $\cA$. It follows from Theorem \ref{thm:faithful} that the resulting tropicalization has a cover by dense open subsets, each of which has a unique continuous section defined on it. By uniqueness, these sections must agree on overlaps, and so we see that this tropicalization is also faithful.

\begin{cor}
If $\fM_{\cT}$ is a hypertoric variety associated to a zonotopal tiling, as in \cite{AP16}, then there is a unique continuous section $s \colon \Trop(\fM_{\cT}) \to \fM_{\cT}^{\an}$ of tropicalization.
\end{cor}

\begin{center}
\noindent\rule{4cm}{.5pt}
\vspace{.25cm}

\noindent {\sc \small \author}\\
{\small Department of Mathematics, University of Oregon, Eugene OR 97403} \\
email: {\href{mailto:\myemail}{\nolinkurl{\myemail}}}
\end{center}


\begin{thebibliography}{CHMR14}

\bibitem[AP16]{AP16}
M.~Arbo and N.~Proudfoot, \emph{Hypertoric varieties and zonotopal tilings}, to appear in Int. Math. Res. Not., Preprint available at \url{http://arxiv.org/abs/1511.09138}.

\bibitem[AK06]{AK06}
F.~Ardila and C.~Klivans, \emph{The Bergman complex of a matroid and phylogenetic trees}, J. Comb. Theory Ser. B \textbf{96} (2006), no. 1, 38--49.


\bibitem[BPR11]{BPR11}
M.~Baker, S.~Payne, and J.~Rabinoff, \emph{Nonarchimedean geometry, tropicalization, and metrics on curves}, 2011, Preprint available at \url{http://arxiv.org/abs/1104.0320}.

\bibitem[B90]{B90}
V.~G.~Berkovich, \emph{Spectral theory and analytic geometry over non-Archimedean fields}, Mathematical Surveys and Monographs, vol. 33, American Mathematical Society, Providence, RI, 1990.

\bibitem[BD00]{BD00}
R.~Bielawski and A.~S.~Dancer, \emph{The geometry and topology of toric kyperk\"{a}hler manifolds}, Comm. Anal. Geom., \textbf{8}(4) (2000), 727--760.

\bibitem[CHMR14]{CHMR14}
R.~Cavalieri, S.~Hampe, H.~Markwig, and D.~Ranganathan, \emph{Moduli spaces of rational weighted stable curves and tropical geometry}, 2014, Preprint available at \url{http://arxiv.org/abs/1404.7426}.


\bibitem[CHW14]{CHW14}
M.~A.~Cueto, M.~H\"{a}bich, and A.~Werner, \emph{Faithful tropicalization of the Grassmannian of planes}, Math. Ann. \textbf{360} (2014), no. 1-2, 391--437.


\bibitem[DP14]{DP14}
J.~Draisma and E.~Postinghel, \emph{Faithful tropicalisation and torus actions}, 2014, Preprint available at \url{http://arxiv.org/abs/1404.4715}.

\bibitem[FGP14]{FGP14}
T.~Foster, P.~Gross, and S.~Payne, \emph{Limits of tropicalizations}, Israel J. Math. \textbf{201} (2014), no. 2, 835--846.

\bibitem[GM10]{GM10}
A.~Gibney and D.~Maclagan, \emph{Equations for Chow and Hilbert quotients}, Journal of Algebra and Number Theory, \textbf{4} (2010), no. 7, 855--885.

\bibitem[G15]{G15}
A.~Gross, \emph{Intersection theory on linear subvarieties of toric varieties}, \emph{Collect. Math.} \textbf{66} (2015), no. 2, 175--190.

\bibitem[GRW14]{GRW14}
W.~Gubler, J.~Rabinoff, and A.~Werner, \emph{Skeletons and tropicalizations}, 2014, Preprint available at \url{http://arxiv.org/abs/1404.7044}.

\bibitem[GRW15]{GRW15}
W.~Gubler, J.~Rabinoff, and A.~Werner, \emph{Tropical skeletons}, 2015, Preprint available at \url{http://arxiv.org/abs/1508.01179}.

\bibitem[HS02]{HS02}
T.~Hausel and B.~Sturmfels, \emph{Toric hyperk\"{a}hler varieties}, Doc. Math. \textbf{7} (2002), 495--534.

\bibitem[MS15]{MS15}
D.~Maclagan and B.~Sturmfels, \emph{Introduction to Tropical Geometry}, Graduate Studies in Mathematics, vol. 161, American Mathematical Society, Providence, RI, 2015.


\bibitem[OR13]{OR13}
B.~Osserman and J.~Rabinoff, \emph{Lifting non-proper tropical intersections}, Contemp. Math. \textbf{605}, Amer. Math. Soc., Providence, RI (2013), 15--44.

\bibitem[P09]{P09}
S.~Payne, \emph{Analytification is the limit of all tropicalizations}, Math. Res. Lett. \textbf{16} (2009), no. 3, 543--556.

\bibitem[P06]{P06}
N.~Proudfoot, \emph{A survey of hypertoric geometry and topology}, in Toric Topology, Contemporary Mathematics vol. 460, AMS, Providence, RI, 2006

\bibitem[PW07]{PW07}
N.~Proudfoot and B.~Webster, \emph{Arithmetic and topology of hypertoric varieties}, J. Alg. Geom. \textbf{16} (2007), 39--63.

\bibitem[R15]{R15}
D.~Ranganathan, \emph{Moduli of rational curves in toric varieties and non-Archimedean geometry}, 2015, Preprint available at \url{http://arxiv.org/abs/1506.03754}.

\bibitem[SS04]{SS04}
D.~Speyer and B.~Sturmfels, \emph{The tropical Grassmannian}, Adv. Geom. \textbf{4} (2004), no. 3, 389--411.

\bibitem[T07]{T07}
J.~Tevelev, \emph{Compactifications of subvarieties of tori}, Amer. J. Math. \textbf{129} (2007), no. 4, 1087Ð1104.


\end{thebibliography}
\end{document}